\documentclass[11pt]{amsart}

\usepackage[dvipsnames]{xcolor}
\usepackage{amssymb,verbatim}
\usepackage{amsmath,amsfonts,enumitem,bm}
\usepackage[mathscr]{euscript} 
\usepackage{amsthm}
\usepackage{url}
\usepackage{graphicx} 
\usepackage{enumitem} 
\usepackage{hyperref} 
\hypersetup{colorlinks} 
\usepackage{bbm} 
\usepackage{bm} 
\usepackage{mathrsfs} 
\usepackage{cancel} 
\usepackage{ytableau} 
\usepackage{mathtools}

\usepackage[colorinlistoftodos]{todonotes}

\RequirePackage{cleveref}
\usepackage{hypcap}
\hypersetup{colorlinks=true, citecolor=darkblue, linkcolor=darkblue}
\definecolor{darkblue}{rgb}{0.0,0,0.7}

\definecolor{darkred}{rgb}{0.68,0,0}

\definecolor{darkgreen}{rgb}{0,.38,0}

\setlist[enumerate]{
	label=\textnormal{({\roman*})},
	ref={\roman*}}

\makeatletter
\def\th@plain{%
	\thm@notefont{}
	\itshape 
}
\def\th@definition{%
	\thm@notefont{}
	\normalfont 
}
\makeatother

\newtheorem{thm}{Theorem}[section]
\newtheorem{lemma}[thm]{Lemma}

\newtheorem*{claim*}{Claim}

\newtheorem{prop}[thm]{Proposition}

\newtheorem{conv}[thm]{Convention}

\theoremstyle{definition}

\newtheorem{rem}[thm]{Remark}

\numberwithin{figure}{section}
\numberwithin{equation}{section}

\def\zz{\mathbb Z}

\def\cc{\mathbb C}

\def\ov{\overline}

\def\ga{\gamma}

\def\al{\alpha}
\def\be{\beta}

\def\cO{\mathcal O}

\def\<{\langle}
\def\>{\rangle}

\def\oa{\overrightarrow}

\def\0{{\mathbf 0}}

\def\.{\hskip.06cm}
\def\ts{\hskip.03cm}

\def\bz{{\textbf{\textit{z}}}}

\def\bx{{\textbf{\textit{x}}}}

\def\by{{\textbf{\textit{y}}}}

\def\bbe{\textbf{\textit{e}}}

\def\bal{{\boldsymbol{\alpha}}}
\def\bbe{{\boldsymbol{\be}}}
\def\bga{{\boldsymbol{\ga}}}

\def\.{\hskip.06cm}
\def\ts{\hskip.03cm}

\newcommand{\textsu}[1]{\textup{\textsf{#1}}}

\newcommand{\CbCla}[1]{\textup{\textsu{#1}}}

\newcommand{\Sigmap}{\ensuremath{\Sigma^{{\textup{p}}}}}

\newcommand{\NP}{\CbCla{NP}}

\newcommand{\BPP}{\CbCla{BPP}}

\newcommand{\coNP}{\CbCla{coNP}}
\newcommand{\E}{\CbCla{E}}
\renewcommand{\P}{\CbCla{P}}

\newcommand{\PH}{\CbCla{PH}}

\newcommand{\PSPACE}{\CbCla{PSPACE}}

\newcommand{\EXPSPACE}{\CbCla{EXPSPACE}}

\newcommand{\AM}{\CbCla{AM}}
\newcommand{\coAM}{\CbCla{coAM}}

\def\GRH{\textup{\sc GRH}}

\def\HN{\textup{\sc HN}}
\def\HNP{\textup{\sc HNP}}
\def\HNPE{\textup{\sc HNPE}}

\def\GWV{\textup{\sc GWVanishing}}

\def\poly{{\P}}

\usepackage{breqn}
\DeclarePairedDelimiter{\angles}{\langle}{\rangle}

\DeclareMathOperator{\Gr}{Gr}

\DeclareMathOperator{\bbC}{\mathbb{C}} 

\newcommand{\bfa}{\mathbf a}
\newcommand{\bfd}{\mathbf d}
\newcommand{\bfU}{\mathbf U}
\newcommand{\bfV}{\mathbf V}
\newcommand{\bfW}{\mathbf W}

\newcommand{\bbP}{\mathbb P}

\newcommand{\hgw}[3]{\angles{ #1 }^{ #2}_{ #3}}

\newcommand{\dSize}{D}

\begin{document}

\title[On Vanishing of GW Invariants]{On Vanishing of Gromov--Witten Invariants}

\author[Igor Pak,  \. Colleen Robichaux, and Weihong Xu]{Igor Pak$^\star$, \.  Colleen Robichaux$^\star$, \. and \. Weihong Xu$^\diamond$}

\makeatletter

\thanks{\today}
\thanks{\thinspace ${\hspace{-.45ex}}^\star$Department of Mathematics,
UCLA, Los Angeles, CA 90095, USA. Email:  \texttt{\{pak,robichaux\}@math.ucla.edu}}
\thanks{\thinspace ${\hspace{-.45ex}}^\diamond$Department of Mathematics,
Caltech, Pasadena, CA 91125, USA. Email:  \texttt{weihong@caltech.edu}}


\begin{abstract}
We consider the decision problem of whether a particular Gromov--Witten invariant 
on a partial flag variety is zero.  We prove that for the $3$-pointed, 
genus zero invariants, this problem is in the complexity class $\AM$ 
assuming the \emph{Generalized Riemann Hypothesis} ($\GRH$),
and therefore lies in the second level of polynomial hierarchy~$\PH$.

For the proof, we construct an explicit system of polynomial equations through
a translation of the defining equations.  We also need to prove an extension
of the \emph{Parametric Hilbert's Nullstellensatz} to obtain our central reduction.
\end{abstract}

\maketitle

\section{Introduction}

\emph{Gromov--Witten invariants} \ts are rational numbers which count particular curves
on a manifold. Genus zero (\(3\)-pointed) Gromov--Witten invariants can be packaged  to
define an associative ring called the (small) \emph{quantum cohomology ring}, which deforms
the ordinary cohomology ring.

Their study began in work of Witten~\cite{Wit91}, arising in the study of string theory on Calabi–Yau manifolds. The mathematical foundations were developed by Gromov~\cite{Gro85}, Kontsevich--Manin~\cite{KM94}, and Ruan--Tian~\cite{RT94}.   Together, this created an active area of study with connections to several other fields. See
\cite{CK99,H+03} and references within, for connections to both mathematical 
and physical aspects of mirror symmetry, \cite{M05} for connections to tropical geometry, 
and \cite{KKPY} for applications in birational geometry.

On a \emph{partial flag variety} $Y$, the genus zero \(3\)-pointed Gromov--Witten (GW) invariants
count the number of degree \(\bf d\) algebraic maps from \(\bbP^1\) to \(Y\) such that a fixed triple of points in $\bbP^1$ lie in
general translates of a given triple of subvarieties $\Omega_i$, respectively.
In this restricted setting, GW invariants and (small) quantum cohomology are
particularly well studied, see e.g.\ \cite{Ful04} and references therein.

While these Gromov--Witten invariants are non-negative integers, we do not, in general, 
have (positive) combinatorial formulas to compute them.
In some cases, such as
when all \(\Omega_i\) are Schubert varieties and one of the cohomology classes \([\Omega_i]\) belongs to a set of generators of the cohomology ring, we obtain such formulas; see, e.g.\ \cite{Ber97,FGP97,CF99b}.
However, when no \([\Omega_i]\) belongs to this set of generators, finding combinatorial formulas for these invariants remains a difficult open problem for general partial flag varieties.

Finding a (positive) combinatorial formula is already a long-standing open problem
in the special case \ts \(\mathbf{d}=0\). In this case,
these invariants recover the \emph{Schubert structure constants}
in the cohomology ring of \(Y\).
That is, degree zero Gromov--Witten invariants count the number of points
in the intersection of a triple of Schubert varieties in general position.
In this setting, the vanishing problem is well-studied, but a combinatorial
criterion to determine their vanishing is not known in full generality.

In recent work of the first two authors \cite{PR-STOC,PR25},
they prove that the problem of deciding the non-vanishing of Schubert structure constants lies in the class $\AM$, assuming the $\GRH$.
See \cite{PR24b} for background, discussion, and references.  Our main result is a direct generalization:

\begin{thm}[{\rm Main Theorem}{}]
\label{t:main-AM}
The problem of deciding if a given genus zero \(3\)-pointed Gromov--Witten invariant is nonzero in the partial flag variety
is in \ts $\AM $ \ts assuming the \ts $\GRH$.
\end{thm}
Here $\AM$ is the class of decision problems whose ``yes" answers can be decided in polynomial time using an \emph{Arthur--Merlin protocol} with two messages, see \cite{AB09, Gol08}.  Heuristically, one should think of \ts $\AM$ \ts as a certain probabilistic extension of the class~$\NP$.
The \emph{Generalized Riemann Hypothesis} ($\GRH$) states that
all nontrivial zeros of $L$-functions \ts $L(s,\chi_k)$ \ts have real part~$\frac12$.
Following Koiran \cite{Koiran96}, this assumption is needed to ensure that there
are enough primes in short arithmetic progressions.

The paper is structured as follows.
In Section~\ref{sec:partial}, we define Gromov--Witten invariants for partial flag varieties and reduce them to those for complete flag varieties. We then give explicit
conditions that completely characterize the non-vanishing of GW invariants for complete flag varieties. In Section~\ref{sec:intro-HN}, we discuss the relevant computational complexity background for the problem and give a new extension of Koiran's theorem (Theorem~\ref{t:main-HNPE}).
In Section~\ref{sec:constrAndPf}, we translate the conditions for non-vanishing into a system of polynomial equations and prove our main theorem. Lastly, in Section~\ref{sec:finRem} we discuss extensions of our result and connections to existing literature.

\medskip

\section{Gromov--Witten invariants on partial flag variety}\label{sec:partial}
In this section, we introduce Gromov--Witten invariants for partial flag varieties. We provide a characterization for the non-vanishing of Gromov--Witten invariants for complete flag varieties. Lastly, we
reduce the vanishing problem for partial flags to the case of complete flags via Woodward's comparison formula \cite{Woo05}.

\subsection{The partial flag variety and GW invariants}\label{sec:geomIntro}
Let \(\bfa=(a_1,\dots, a_k)\in\mathbb{Z}^k\), where we take $0< a_1< \ldots< a_k< n$. The \emph{partial flag variety}, denoted \(F(\bfa,n)\), consists of flags of subspaces $V_\bullet$:
$$\{0\}=V_0\subset V_1 \subset V_2 \subset ... \subset V_{k}
\subset V=\bbC^n, \text{ where } \dim V_i=a_i.$$
We say the flag $V_\bullet$ above is a {complete flag} in $\bbC^n$ if $k=n$. Note that we can associate each vector space $V_i$ with an $n\times i$ matrix whose columns are the $i$ vectors generating $V_i$.

Restricting to particular $\bfa$ produces several well-studied spaces. For example, the \emph{Grassmannian} $\Gr(k,n)$ of $k$-dimensional planes in $\mathbb{C}^n$ is the partial flag variety $F((k),n)$.  Additionally, the \emph{complete flag variety} $F(n)$, which contains all complete flags in $\mathbb{C}^n$, is the partial flag variety $F((1,2,\ldots,n-1),n)$.

For $n\in\mathbb{Z}$, define $[n]:=\{i\in{\mathbb{Z}} \, : \, 1\leq i\leq n\}$. Let $S_n$ denote the symmetric group on $[n]$.
We consider the \emph{Young subgroup} with respect to $\bfa$:
\[
S_\bfa=S_{a_1}\times S_{a_2-a_1}\times \dots \times S_{n-a_k}\subseteq S_n.
\]
For a fixed complete flag $E_\bullet$,
we define the \emph{Schubert variety} $Y_{[w]}(E_\bullet)$ corresponding to a coset \([w]\in S_n/S_\bfa\) as
\[
Y_{[w]}(E_\bullet)=\{ V_\bullet\in F(\bfa,n): {\rm{rank}}(E_j\to \bbC^n/V_i)\leq r_w(i,j)\ts \text{ for }\ts i\in[k] \text{ and }\ts j\in[n]\ts\}.
\]
Here \(w\in S_n\) is any representative of \([w]\) and $r_w:[n]^2\rightarrow [n]$ is the rank function  such that
\[r_w(i,j)=|\{h\leq i \, : \, w(h)\leq j\}|.\]
The cohomology class of \(Y_{[w]}\), denoted  $[Y_{[w]}]\in H^*(F(\bfa,n))$, is independent of the choice of the reference flag
$E_\bullet$.

Let \(\bfd=(\mathrm{d}_1,\dots,\mathrm{d}_k)\in\mathbb{Z}_{\geq 0}^k\) and cosets \([u],[v],[w]\in S_n/S_\bfa\). We define the \emph{Gromov--Witten (GW) invariant} \(\hgw{[Y_{[u]}],[Y_{[v]}],[Y_{[w]}]}{}{\bfd}\) as the number of maps \(f: \bbP^1\to F(\mathbf{a},n)\) such that:
\begin{enumerate}
    \item \(f(0)\in Y_{[u]}(E_\bullet)\), \(f(1)\in Y_{[v]}(E'_\bullet)\), and \(f(\infty)\in Y_{[w]}(E''_\bullet)\), for \(E_\bullet, E'_\bullet, E''_\bullet\) complete flags in general position, and
    \item there are exactly \(\mathrm{d}_i\) pre-images under $f$ of a general hyperplane pulled back along the composition
    \[
    F(\mathbf{a},n)\to \Gr(a_i,n)\to \bbP^{\binom{n}{a_i}-1},
    \]
    for each \(i\in[k]\).
\end{enumerate}
If there are infinitely many such maps, then \(\hgw{[Y_{[u]}],[Y_{[v]}],[Y_{[w]}]}{}{\bfd}\) is defined to be \(0\).

Define the decision problem
$$
\GWV \, := \, \big\{\hgw{[Y_{[u]}],[Y_{[v]}],[Y_{[w]}]}{}{\bfd}=^? 0 \big\}.
$$

\subsection{Characterizing GW invariants on complete flag varieties}\label{sec:complete}

In Section~\ref{sec:partRed}, we will see that computing GW invariants on partial flag varieties can be reduced to computing those on complete flag varieties.
Therefore, we now focus on the case of the complete flag variety \(F(n)\). We use standard definitions and notation in algebraic geometry. We refer to \cite{Vak25} for necessary background. We follow the setup of \cite{CF99a}.

Write $V_X=V\otimes \cO_X$ for any scheme $X$.
Then there is a universal sequence of quotient bundles on $F(n)$:
$$V_{F(n)}\twoheadrightarrow Q_{n-1}\twoheadrightarrow\cdots\twoheadrightarrow Q_{1}.$$
Here $\mbox{rank }Q_i=i$, and each $Q_i\twoheadrightarrow Q_{i-1}$ is a surjection.

Fix a flag $E_\bullet\in F(n)$.
In \(F(n)\), Schubert varieties are indexed by $w\in S_n$. We can equivalently define them as follows:
$$X_w(E_\bullet)=\{V_\bullet\in F(n) \ts : \ts \mbox{rank}_{V_\bullet}(E_j\otimes
\cO_{F(n)}\twoheadrightarrow Q_i)\leq r_w(i,j)\ts \mbox{ for } \ts i,j\in[n]\}.$$
The Schubert variety $X_w(E_\bullet)$ is a codimension $\ell(w)$ subvariety in $F(n)$, where $\ell(w)$ is the number of inversions
of $w$.  {The cohomology class of \(X_w\), denoted by $[X_w]$ is independent of the choice of the flag
$E_\bullet$.}

Fix \(\bfd=(\mathrm{d}_1,\dots,\mathrm{d}_{n-1})\in\mathbb{Z}_{\geq 0}^{n-1}\). The hyperquot scheme \(H_{\bfd}(n)\) \cite{CF99a,Kim96} parametrizes sequences of successive quotients of sheaves on \(\bbP^1\):
\begin{align*}
    V^*_{\bbP^1}\twoheadrightarrow B_{n-1}\twoheadrightarrow \dots \twoheadrightarrow B_1.
\end{align*}
 Here \(V^*\) denotes the vector space dual to \(V\). Additionally, each \(B_i\) has rank \(i\) and degree \(\mathrm{d}_{n-i}\).
The scheme \(H_{\bfd}(n)\) is a smooth irreducible variety of dimension
\[\binom{n}{2}+2\sum_{i=1}^{n-1}\mathrm{d}_i.\]
Further, \(H_{\bfd}(n)\) contains \(M_\bfd(n)\), the set of degree \(\bfd\) morphisms from \(\bbP^1\) to \(F(n)\), as a dense open subset.

Given \(u,v,w\in S_n\), the associated GW invariant \(\hgw{[X_u],[X_v],[X_w]}{}{\bfd}\) is \(0\) unless the codimensions of \(X_u, X_v, X_w\) sum up to the dimension of \(M_\bfd(n)\), i.e.,
\begin{align}\label{eqn:dimcount}
    \ell(u)+\ell(v)+\ell(w)=\binom{n}{2}+2\sum_{i=1}^{n-1}\mathrm{d}_i.
\end{align}
See~\cite[$\S$4]{CF99a} for further discussion.

\begin{lemma}\label{lem:GWsetup}
Consider $u,v,w\in S_n$ and \(\bfd=(\mathrm{d}_1,\dots,\mathrm{d}_{n-1})\).
Then for generic $U_\bullet, V_\bullet, W_\bullet\in F(n)$, the GW invariant
$\hgw{[X_u],[X_v],[X_w]}{}{\bfd}\neq 0$ if and only if each of the following conditions is satisfied:
\begin{enumerate}
    \item \eqref{eqn:dimcount} holds.
    \item For each $1\leq i\leq h\leq n-1$, there exist non-negative integers $d_{h,i}$  such that \(\sum_i d_{h,i}=\mathrm{d}_h\).
    \item Take \(d_{n,j}=0\) for all \(j\in[n]\). Then for each \(h\in[n-1]\), there exist \(h\times (h+1)\) matrices \(M_h(s,t)\) over \(\bbC[s,t]\)  whose \((i,j)\)-th entry is either \(0\) or homogeneous of degree \(d_{h,i}-d_{h+1,j}\) such that
    \begin{enumerate}
    \item[(a)] \(M_h(s,t)\) is full rank over \(\bbC(s,t)\) for all \(h\in[n-1]\), and
    \item[(b)]\label{item:b} for \(i,j\in[n]\), each \(i\times j\) product matrix below \[M_i(0,1)\dots M_{n-1}(0,1)\bfU_j,\quad M_i(1,1)\dots M_{n-1}(1,1)\bfV_j,\quad M_i(1,0)\dots M_{n-1}(1,0)\bfW_j,\] has rank \textcolor{purple}{ at most } \(r_u(i,j), r_v(i,j), r_w(i,j)\), respectively.
\end{enumerate}
\end{enumerate}
Here, \(\bfU_j\, \bfV_j, \bfW_j\) are the matrices corresponding to the vector spaces \(U_j, V_j, W_j\),
respectively.
\end{lemma}
 \begin{proof}
As noted above, if \eqref{eqn:dimcount} doesn't hold, then \(\hgw{[X_u],[X_v],[X_w]}{}{\bfd}=0\) for dimension reasons.
Now assume \eqref{eqn:dimcount} holds.

Let \(e_0, e_1, e_\infty: M_\bfd(n)\to F(n)\) be the maps given by evaluations at \(0,1,\infty\in\bbP^1\).
Fix \(a\in \bbP^1\) and a flag $V_\bullet\in  F(n)$. Let \(X'_w(V_\bullet,a)\subseteq H_{\bfd}(n)\) be the degeneracy locus corresponding to the sequences of quotients of sheaves
\begin{align}\label{eqn:seq}
    V^*_{\bbP^1}\twoheadrightarrow B_{n-1}\twoheadrightarrow \dots \twoheadrightarrow B_1
\end{align}
such that
$$\mbox{rank\,}_a
(V_j\otimes \cO_{\bbP^1}\rightarrow A_i^*)\leq r_w(i,j),\quad \text{for }\ts
i,j\in[n],$$
where \(A_i\) is the kernel of the map \(V^*_{\bbP^1}\twoheadrightarrow B_{n-i}\).  Recall that each \(B_i\) has rank \(i\) and degree \(d_{n-i}\).

 Then by \cite[$\S$4]{CF99a},  \(\hgw{[X_u],[X_v],[X_w]}{}{\bfd}\) counts the number of points in
\begin{align}\label{eqn:quot-map}
    X'_u(U_\bullet,0)\cap X'_v(V_\bullet,1)\cap X'_w(W_\bullet,\infty)=e_0^{-1}\left(X_u(U_\bullet)\right)\cap e_1^{-1}\left(X_v(V_\bullet)\right)\cap e_\infty^{-1}\left(X_w(W_\bullet)\right).
\end{align}
 Thus, \(\hgw{[X_u],[X_v],[X_w]}{}{\bfd}=0\) if and only if \(X'_u(U_\bullet,0)\cap X'_v(V_\bullet,1)\cap X'_w(W_\bullet,\infty)=\emptyset\).
 When \(U_\bullet, V_\bullet, W_\bullet\) are in general position, this intersection~\eqref{eqn:quot-map} consists of reduced points. Each such point corresponds to a sequence of quotients in $H_{\bfd}(n)$ as in~\eqref{eqn:seq} such that
\begin{align}\label{eqn:rkEq}
\begin{split}
    \mbox{rank\,}_0
(U_j\otimes \cO_{\bbP^1}\rightarrow A_i^*)\leq r_w(i,j),\quad
\text{ for } & \ts i,j\in[n],\\
\mbox{rank\,}_1
(V_j\otimes \cO_{\bbP^1}\rightarrow A_i^*)\leq r_w(i,j),\quad
\text{ for }  &\ts i,j\in[n], \text{ and}\\
\mbox{rank\,}_\infty
(W_j\otimes \cO_{\bbP^1}\rightarrow A_i^*)\leq r_w(i,j),\quad
\text{ for } & \ts i,j\in[n].
\end{split}
\end{align}

For each $h\in[n-1]$, the sheaf \(A_h\) is a subsheaf of the free sheaf \(V^*_{\bbP^1}\), so each \(A_h\) is locally free. Since \(B_h\) has rank \(h\) and degree \(\mathrm{d}_{n-h}\), each \(A_h\)
must have rank \(h\) and degree \(-\mathrm{d}_h\). Therefore, each \(A_h\) is isomorphic to \(\oplus_{i=1}^h\cO(-d_{h,i})\) for some non-negative integers \(d_{h,i}\) such that \(\sum_i d_{h,i}=\mathrm{d}_h\). The existence of such integers is encoded in (ii).
The data of \eqref{eqn:seq} are equivalent to the data of a sequence of inclusions of sheaves
\[
A_1\hookrightarrow \dots\hookrightarrow A_{n-1}\hookrightarrow V_{\bbP^1}^*.
\]
Note that a map \(\cO(a)\to \cO(b)\)  is determined by a section of \(\cO(b-a)\). In addition, sections of \(\cO(b-a)\) can be identified with polynomials in \(\bbC[s,t]\) that are homogeneous of degree \(b-a\). Note that when \(b<a\) this line bundle has no non-zero section. Set \(A_n=V^*_{\bbP^1}\cong \cO(0)^{\oplus n}\).
Combining these facts, each map \(A_h\to A_{h+1}\)  is determined by an \(h\times (h+1)\) matrix \(M_h(s,t)\) over \(\bbC[s,t]\) for \(h\in[n-1]\) whose \((i,j)\)-th entry is either \(0\) or homogeneous of degree \(d_{h,i}-d_{h+1,j}\). Finally, condition (iii)(a) encodes the injectivity of the maps $A_h\to A_{h+1}$, and condition (iii)(b) imposes \eqref{eqn:rkEq}.
 \end{proof}

Note that, without loss of generality, we may assume \(d_{i,1}\leq \dots\leq d_{i,i}\) for \(i\in[n-1]\) above.

\subsection{Reduction from partial flag varieties to complete flag varieties}\label{sec:partRed}
First set \(\bfa=(a_1,\dots, a_k)\in\mathbb{Z}^k\), where $1\leq a_1< \ldots <a_k\leq n$. Consider \(\bfd=(\mathrm{d}_1,\dots,\mathrm{d}_k)\in\mathbb{Z}_{\geq 0}^k\) and cosets \([u],[v],[w]\in S_n/S_\bfa\).

By \cite[Lemma/Def.~1]{Woo05}, there is a unique sequence of non-negative integers \(\widehat{\bfd}=(\widehat{d}_1,\dots,\widehat{d}_{n-1})\) such that
\(\widehat{d}_{a_i}=\mathrm{d}_i\) for \(i\in[k]\)
and
\begin{equation}\label{eq:lift}
    -\widehat{d}_{i-1}+\widehat{d}_i+\widehat{d}_j-\widehat{d}_{j+1}\in \{0,-1\}
\end{equation}
when \(a_h<i\leq j<a_{h+1}\) for some \(h\in[k]\). Here we set \(a_0=0,\  a_{k+1}=n\), and \(\widehat{d}_0=\widehat{d}_n=0\). Moreover, as a special case of \cite[Thm~2]{Woo05}, we have
\begin{equation}\label{eq:partialReduction}
    \hgw{[Y_{[u]}],[Y_{[v]}],[Y_{[w]}]}{}{\bfd}=\hgw{[X_{\widehat{u}}],[X_{\widehat{v}}],[X_{\widehat{w}w'}]}{}{\widehat{\bfd}},
\end{equation}
where \(\widehat{u},\widehat{v},\widehat{w}\) are the minimal length representatives of the cosets \([u],[v],[w]\), respectively. Here \(w'\) is the longest element in the subgroup of \(S_\bfa\) generated by transpositions \((i,j+1)\) such that
\[
-\widehat{d}_{i-1}+\widehat{d}_i+\widehat{d}_j-\widehat{d}_{j+1}=0.
\]
\begin{rem}
    Alternatively, instead of going through this reduction to complete flag varieties,
    one could instead generalize Lemma~\ref{lem:GWsetup} to partial flag varieties and proceed directly.
\end{rem}

\medskip

\section{Hilbert's Nullstellensatz and variants}\label{sec:intro-HN}
Take $R=\cc{}[x_1,\dots,x_s]$ for some $s>0$. Set $\bx=(x_1,\ldots,x_s)$.
\emph{Hilbert's weak Nullstellensatz} \ts
 states that a polynomial system
\begin{equation}\label{eq:HN-system}
f_1(\bx) \. = \. \ldots \. =  \. f_m(\bx) \. = \. 0,
\end{equation}
where $f_i\in R$, is unsatisfiable over $\cc$ if and only if
there exist \ts $g_1,\ldots,g_m\in R$ such that
$$
\sum_{i=1}^m \. f_i(\bx) \. g_i(\bx) \, = \, 1\ts.
$$

Now, impose that \ts $f_1,\ldots,f_m\in \zz[x_1,\dots,x_s]$.
The decision problem $\HN$ (\emph{Hilbert's Nullstellensatz})
queries if the system of equations \eqref{eq:HN-system}
is satisfiable over \ts $\mathbb{C}$.

For $g\in \zz[x_1,\ldots,x_n]$, let $\deg(g)$ denote the degree
of~$g$, and let $s(g)$ denote the sum of bit-lengths of coefficients in~$g$.
We define the \emph{size} of~$g$ to be
$$
\phi(g) \. := \. \deg(g) \. + \. s(g).
$$
Then the \emph{size} of the system of polynomials  $\ov f = (f_1,\ldots, f_m)$ is defined as
$$
\phi\big(\ov f\big) \. := \. \sum_{i=1}^{m} \. \deg(f_i) \. + \. \sum_{i=1}^{m} \. s(f_i).
$$
For a matrix $M$ with polynomial entries, the \emph{size} $\phi(M)$
is the sum of the sizes of its entries.

Work of Mayr and Meyer \cite{MM82} implies that \ts $\HN$ \ts is decidable.
In particular, they show \ts $\HN$ \ts is in \ts $\EXPSPACE$, and that \ts $\HN$ \ts is \ts $\NP$-hard.
Later work of Brownawell \cite{Bro87} and Koll\'ar \cite{Kol88} for the \emph{effective Nullstellensatz}
proves the existence of $g_i$ with single exponential size. These improved bounds place \ts $\HN$ \ts in \ts $\PSPACE$.
Then in a landmark paper, Koiran proved that
\ts $\HN$ \ts is in the polynomial hierarchy:

\smallskip

\begin{thm}[{\cite[Thm~2]{Koiran96}}{}]\label{t:main-HN}
    \. $\HN$ \ts is in \ts $\AM$ \ts assuming \ts $\GRH$.
\end{thm}

\smallskip
Above, $\GRH$ refers to the \emph{Generalized Riemann Hypothesis}.
Koiran uses this $\GRH$ assumption to ensure the existence of primes in certain intervals with particular modular constraints.

For our purposes, we need the following
strengthening of Theorem~\ref{t:main-HN}.
Let
$$
f_1,\ldots,f_m\.\in\. \mathbb{Q}(y_1,\ldots,y_k) \ts [x_1,\dots,x_s]\ts.
$$
The decision problem $\HNP$ (\emph{Parametric Hilbert's Nullstellensatz})
asks if the polynomial system \eqref{eq:HN-system} has a solution over
\ts $\overline{\mathbb{C}(y_1,\ldots,y_k)}$.
In recent work, Ait El Manssour, Balaji, Nosan, Shirmohammadi and Worrell
extended Theorem~\ref{t:main-HN} to~$\HNP:$

\smallskip

\begin{thm}[{\cite[Thm~1]{A+25}}{}]\label{t:main-HNP}
    \. $\HNP$ \ts is in \ts $\AM$ \ts assuming \ts $\GRH$.
\end{thm}

\smallskip

See \cite{A+25} for extensive background of
\ts $\HNP$ \ts and other related work.

\subsection{A further generalization}\label{ss:hnpExt}
Our construction needs an additional refinement of $\HN$ and $\HNP$, which we call $\HNPE$ (\emph{Parametric Hilbert's Nullstellensatz with Exponents}).

We consider expressions $f_i$ of the form:
\begin{equation}\label{eq:f_HNPE}
    f_i=\sum_{\al=(\al_1,\ldots,\al_s)}p_{\al}x^\al,
\end{equation}
where $p_{\al}\in \mathbb{Q}(y_1,\ldots,y_k)[x_1,\ldots,x_s]$ and $\al_j\in\{0\}\cup\{z_1,\ldots,z_\ell\} $ for $j\in[s]$. That is, variables may appear in the exponents in $f_i$. Let ${\bf z}=(z_1,\ldots,z_\ell)$ be a vector of those exponent variables.
Let $f_i(\oa z)$ denote the resulting $f_i$ after evaluating  ${\bf z}$ at some $\oa z\in\mathbb{Z}^\ell$.

Fix a system $\{f_i=0\}_{i\in[m]}$ with $f_i$ as in~\eqref{eq:f_HNPE}, some $M\in\mathbb{Z}_{>0}$, row vectors $a_i\in\mathbb{Z}^{\ell}$, and constants $b_i \in\mathbb{Z}$ for indices $i\in I\subset\mathbb{Z}_{\geq 0}$. Suppose $|I|= O(n^c)$ for some $c\geq 0$.
The decision problem $\HNPE$
asks if there exists an integer evaluation $\oa z$ of $\bz$ such that
\begin{enumerate}
    \item $a_i\cdot{\oa z}\leq b_i$,
    \item $0\leq {\oa{z_j}} \leq M$
    for each $j\in [\ell]$, and
     \item $\{f_i(\oa z)=0 \}_{i\in [m]}$ has a solution over $\overline{\mathbb{C}(y_1,\ldots,y_k)}$.
\end{enumerate}

\begin{thm}\label{t:main-HNPE}
    \. $\HNPE$ \ts is in \ts $\AM$ \ts assuming \ts $\GRH$.
\end{thm}

\begin{proof}
  Consider the system $\mathcal{S}$ defined by $\{f_i=0\}_{i\in[m]}$, where $f_i$ are of the form~\eqref{eq:f_HNPE}.
   Our inputs are $\mathcal{S}$, $\{a_i\}_{i\in I}$, $\{b_i\}_{i\in I}$, and $M$.
 Note that we assume $\AM$ protocols satisfy perfect completeness. Restated, if a decision problem is in $\AM$ and the answer of an instance of the decision problem is ``yes", the $\AM$ protocol will output ``yes" with probability $1$ in that instance.

  First, query to determine a non-negative integer evaluation $\oa z$ of $\bf z$ such that the evaluated system, denoted $\mathcal{S}(\oa z)$, is satisfiable over
\ts $\overline{\mathbb{C}(y_1,\ldots,y_k)}$ while satisfying the added constraints (i) and (ii) for $\oa{z}$.
By (ii), $\oa z$ has polynomial size.

Now we appeal to the $\AM$ protocol of \cite{A+25} since the satisfiability of $\mathcal{S}(\oa z)$ is an instance of $\HNP$.
If the desired $\oa z$ exists, our query will return $\oa z\in\mathbb{Z}^{\ell}$ such that (i), (ii), and (iii) hold.
We can check (i) and (ii) for each $i\in I$ in polynomial time since $M$ bounds the sizes of $z_j$ and $|I|=O(n^c)$.
Then to check (iii), we use the $\AM$ protocol for $\HNP$ on $\mathcal{S}(\oa z)$ to output ``satisfiable" with probability $1$.

If no satisfiable $\oa z$ exists, our query may return some $\oa z\not\in\mathbb{Z}^{\ell}$ or some $\oa z\in\mathbb{Z}^{\ell}$ such that (i) or  (ii) fails.
We can detect errors of this form in polynomial time to output ``unsatisfiable" with probability $1$.  Alternatively our query may return some $\oa z\in\mathbb{Z}^{\ell}$ where (i) holds,  (ii) holds,
but $\mathcal{S}(\oa z)$ is unsatisfiable.
In this case,
the $\AM$ protocol  for $\HNP$  on $\mathcal{S}(\oa z)$ will output ``satisfiable" with probability at most  $\frac{1}{2}$.

Thus $\HNPE$ is in $\AM[4]$, the class of decision problems whose ``yes" answers can be decided in polynomial time using an {Arthur--Merlin protocol} with four messages.
Since $\AM=\AM[4]$, see \cite[Thm~2.1]{Bab85}, the result follows.
\end{proof}

\subsection{Reduction to $\HNPE$}\label{ss:intro-sketch}
We prove our Main~Theorem~\ref{t:main-AM} by showing that GW vanishing
is an instance of $\HNPE$.

\begin{lemma}[{\rm Main Lemma}{}] \label{l:GWV-HNP}
\. $\neg\GWV$ \. reduces to \. $\HNPE$. \.
\end{lemma}
The proof of Main~Lemma~\ref{l:GWV-HNP} is given in~$\S$\ref{ss:mainPf}.

\medskip

\section{Main construction and proof} \label{sec:constrAndPf}
In this section we translate the conditions given in Lemma~\ref{lem:GWsetup} (ii) and (iii) into an explicit system of polynomial equations.

\subsection{The construction}\label{ss:constr}
 Let \(\bfa=(a_1,\dots, a_k)\in\mathbb{Z}^k\), where $1\leq a_1< \ldots <a_k\leq n$.
 Suppose that
$\bfd=(\mathrm{d}_1,\dots,\mathrm{d}_k)\in\mathbb{Z}_{\geq 0}^k$ for some $k\in[n]$.
Set $\dSize=\sum_{i\in [k]} { \mathrm{d}_i}$.

 First, we describe a system for the constraints to determine integers that satisfy Lemma~\ref{lem:GWsetup} (ii).
 Let ${\textbf{\textit{d}}}$ denote a set of variables
 $d_{i,j}$ for $1\leq j\leq i\leq n-1$ and variables $\widehat{d_{i}}$ for $i\in[n-1]$.
 Take \(a_0=0,\  a_{k+1}=n\), and \(\widehat{d}_{0}=\widehat{d}_{n}=0\).
Define $\mathcal {T}({\bfd})$ as the system formed by the constraints:
$$
\left\{ \ \aligned
&  \widehat{d}_{a_h}  \. = \. \mathrm{d}_h &\text{ for }\. h\in[k].\\
&  -1\leq-\widehat{d}_{i-1}+\widehat{d}_{i}+\widehat{d}_{j}-\widehat{d}_{j+1}\leq 0 &\text{ for }\. a_h<i\leq j<a_{h+1}, \ h\in[k].\\
&  \Big(\sum_{j=1}^{i}d_{i,j}\Big)-\widehat{d}_{i}=0 &\text{ for }\. i\in[n-1].\\
&  \widehat{d}_{i}\geq 0 & \text{ for }\. i\in[n-1].\\
&  d_{i,j}\geq 0, &\text{ for }\. 1\leq j\leq i\leq n-1.\\
&  d_{i,j}-d_{i,j+1}\leq 0 &\text{ for }\. 1\leq j\leq i\leq n-1, \text{ and } j<n-1.
\endaligned
\right.
$$
First, note that through a straightforward induction, the first, second, and fourth equations enforce  $0\leq \widehat{d}_{i} \leq \dSize$ for each $i\in[n-1]$.
Combining this with the third and fifth equations implies $0\leq d_{i,j} \leq \dSize$ for each $1\leq j\leq i \leq n-1$.
 Coefficients in $\mathcal {T}({\bfd})$ are all in $\{0,\pm1\}$ and constants in $\mathcal {T}({\bfd})$ are in $\{0,\pm1\}\cup \{ \mathrm{d}_1,\ldots, \mathrm{d}_k\}$. Since $k\leq n$, $\mathcal {T}({\bfd})$ requires $I=O(n^2)$ equations.

Let $\oa d$ be some integer evaluation of variables ${\textbf{\textit{d}}}$. Define $\oa{d_{n,j}}:=0$ for all $j\in[n-1]$.
Take $s$ and $t$ to be variables.
Now we build matrices $M_h(s,t)$ to address Lemma~\ref{lem:GWsetup} (iii).
For each $h\in[n-1]$,
define the $h\times (h+1)$ matrix of polynomials $M_{h}(s,t)$ with $i,j$-th entry:
\[p^{(h)}_{ij}(s,t)=\sum_{m=0}^{L}a_{ijm}^{(h)}s^mt^{L-m}, \text{ for } L=\oa{d_{h,i}}-\oa{d_{h+1,j}}.\]
Here $a_{ijm}^{(h)}$ are also treated as variables. Let ${\textbf{\textit{a}}}$ denote the set of variables $a_{ijm}^{(h)}$ where $h\in[n-1]$
and  $0\leq m\leq \oa{d_{h,i}}-\oa{d_{h+1,j}}$.
Then $|{\textbf{\textit{a}}}|=O(n\cdot n^2\cdot \dSize)=O(\dSize \cdot n^3)$.
For succinctness, we may write $M_h=M_h(s,t)$. Note that entries $p^{(h)}_{ij}(s,t)$ in $M_h$ have size $O(\dSize)$.

We now construct additional matrices for Lemma~\ref{lem:GWsetup} (iii)(a).
For each $h\in[n-1]$ define ${B_{h}}=(b_{ij}^{(h)})$
to be an $h \times h$ matrix of variables. Similarly define ${C_{h}}=(c_{ij}^{(h)})$
to be an $(h+1)\times (h+1)$ matrix of variables.
Define the sets of variables \ts ${\textbf{\textit{b}}}=\{b_{ij}^{(h)}\}$ where $1\leq i,j\leq h\leq n$
and \ts ${\textbf{\textit{c}}}=\{c_{ij}^{(h)}\}$ where $1\leq i,j\leq h+1\leq n+1$.
 Then $|{\textbf{\textit{b}}}|=O(n^3)$ and $|{\textbf{\textit{c}}}|=O(n^3)$.

Take $\bfU=(\al_{ij}), \bfV=(\be_{ij}), \bfW=(\ga_{ij})$ to be $n\times n$ matrices of parameters. Let $\bal, \bbe, \bga$ denote the sets of parameters therein, respectively.
Let $\bfU_h$ denote the submatrix of $\bfU$ formed by its first $h$ columns. Define $\bfV_h$ and $\bfW_h$ analogously.

Now we build matrices to force the constraints in Lemma~\ref{lem:GWsetup} (iii)(b).
For each $i,j\in[n]$ and $\sigma\in\{u,v,w\}$ construct the $i\times r_{\sigma}(i,j)$ matrix of variables $X_{ij}^\sigma=(X_{i,j,p,q}^\sigma)$. Similarly, form the $r_{\sigma}(i,j)\times j$ matrix  of variables $Y_{ij}^\sigma=(y_{i,j,p,q}^\sigma)$. Let $\bx$ and $\by$ denote the sets of variables appearing in these matrices, respectively. Then $|\bx|=|\by|=O(n^4)$ since $r_\sigma(i,j)\leq n$.
For ease of notation, denote the $i\times j$ matrix
\[R_{ij}({\sigma}):=X_{ij}^\sigma\cdot Y_{ij}^\sigma,\ts \text{ for } \ts  i,j\in[n] \ts\text{ and } \ts\sigma\in\{u,v,w\}. \]
Note entries in $R_{ij}(\sigma)$ have size $O(n^2)$.

Let ${\sf Id}_{h}\cup 0$ denote the $h\times(h+1)$ matrix whose leftmost $h\times h$ submatrix is ${\sf Id}_{h}$ and last column is the $0$ vector.

Let $\mathcal {S}(u,v,w,{\oa d})$ be the system formed by the constraints:
$$
\left\{ \ \aligned
& B_h\cdot  M_h(s,t)\cdot C_h\. = \. {\sf Id}_{h}\cup 0  &\text{ for }\. h\in[n-1].\\
&  R_{n,j}(u) \. = \.  \bfU_j &\text{ for }\. j\in[n].\\
&  R_{n,j}(v)  \. = \.  \bfV_j &\text{ for }\. j\in[n].\\
&  R_{n,j}(w) \. = \.  \bfW_j &\text{ for }\. j\in[n].\\
&  R_{i-1,j}(u)  \. = \.M_{i-1}(0,1)\cdot R_{ij}(u) &\text{ for }\. i,j\in[n].\\
&  R_{i-1,j}(v)  \. = \.M_{i-1}(1,1)\cdot R_{ij}(v) &\text{ for }\. i,j\in[n].\\
&  R_{i-1,j}(w)  \. = \.M_{i-1}(1,0)\cdot R_{ij}(w) &\text{ for }\. i,j\in[n].
\endaligned
\right.
$$
Here $\mathcal {S}(u,v,w,{\oa d})$ uses variables ${\textbf{\textit{a}}}\cup {\textbf{\textit{b}}}\cup {\textbf{\textit{c}}}\cup \bx\cup\by\cup \bz\cup\{s,t\}$ and parameters $\bal\cup \bbe\cup \bga$.
We see the first constraint has size $O(n\cdot n^2
\cdot \dSize \cdot n^2)=O(\dSize\cdot n^5)$ since entries in $M_h$ have size $O(\dSize)$. The next three equations have size  $O(n\cdot n^2)=O(n^3)$.
The final three constraints have size $O(n^2\cdot n^2\cdot n^2)=O(n^6)$ since entries in $M_h$ are linear for evaluated $s,t\in\{0,1\}$.
Thus $\mathcal {S}(u,v,w,{\oa d})$ has size $O(n^5(n+\dSize))$.

\begin{prop}\label{prop:GWsystem}
Consider $u,v,w\in S_n$ and ${\bfd}\in\zz_{\geq 0}^{k}$ such that~\eqref{eqn:dimcount} holds. Then we have $\hgw{[Y_{[u]}],[Y_{[v]}],[Y_{[w]}]}{}{\bfd}\neq 0$ if and only if for a solution ${\oa d}$ to $\mathcal {T}({\bfd})$ and a generic choice of evaluations $\oa \al, \oa \be, \oa \ga$ of $\bal, \bbe, \bga$, the system $\mathcal {S}(\widehat{u},\widehat{v},\widehat{w}w',{\oa d})$
has a solution over \(\mathbb{C}\),  where $(\widehat{u},\widehat{v},\widehat{w}w')$ are as in~\eqref{eq:partialReduction}.
\end{prop}
\begin{proof}
We examine the conditions comprising $\mathcal{T}(\bfd)$ and $\mathcal {S}(\widehat{u},\widehat{v},\widehat{w}w',{\oa d})$ and show these translate those in Lemma~\ref{lem:GWsetup}.
First, $\mathcal{T}(\bfd)$ directly  gives (ii) using~\eqref{eq:lift}. Using~\eqref{eq:partialReduction}, the problem of determining $ \hgw{[Y_{[u]}],[Y_{[v]}],[Y_{[w]}]}{}{\bfd}\neq 0$ reduces to $\hgw{[X_{\widehat{u}}],[X_{\widehat{v}}],[X_{\widehat{w}w'}]}{}{\widehat{\bfd}}\neq 0$.

We now consider $\mathcal {S}(\widehat{u},\widehat{v},\widehat{w}w',{\oa d})$.
Note that $M_{h}$ is rank $h$ if and only if $B_h\cdot M_{h}(s,t)\cdot C_h={\sf Id}_h\cup 0$ for matrices $B_h,C_h$.
Further note $M_{h}(s,t)$ is rank $h$ for general $s,t$ if and only if $M_{h}(s,t)$ is rank $h$ for some choice of $s,t$.
 Thus the satisfiability of this condition characterizes Lemma~\ref{lem:GWsetup} (iii)(a).

Together, the remaining equations ensure the following:
\begin{itemize}
    \item[$\diamond$] $R_{ij}({\widehat{u}})=M_i(0,1)\dots M_{n-1}(0,1)\bfU_j$,
    \item[$\diamond$] $R_{ij}({\widehat{v}})=M_i(1,1)\dots M_{n-1}(1,1)\bfV_j$, and
    \item[$\diamond$] $R_{ij}({\widehat{w}w'})=M_i(1,0)\dots M_{n-1}(1,0)\bfW_j$
\end{itemize}
The definition of $R_{ij}({\sigma})$ as a product of an $i\times r_{\sigma}(i,j)$ and an $r_{\sigma}(i,j)\times j$ matrix ensures the rank of $R_{ij}({\sigma})$ is at most $r_{\sigma}(i,j)$, where $\sigma\in S_n$. Thus these last six conditions force Lemma~\ref{lem:GWsetup} (iii)(b).
\end{proof}

\subsection{Equivalence of satisfiability}\label{ss:geom}
We outline the following argument for completeness.
\begin{lemma}\label{prop:genericRed}
Consider $u,v,w\in S_n$ and ${\oa d}\in \mathbb{Z}^{(n-1)^2}$.
Then the following are equivalent:
\begin{enumerate}
    \item  $\mathcal {S}(u,v,w,{\oa d})$ has a solution over $\overline{\mathbb{C}(\bal, \bbe, \bga)}$.
    \item $\mathcal {S}(u,v,w,{\oa d})$  has a solution over  $\mathbb{C}$ for
a generic choice of evaluations $\oa \al, \oa \be, \oa \ga$ of $\bal, \bbe, \bga$.
\end{enumerate}
\end{lemma}
\begin{proof}
The system $\mathcal {S}(u,v,w,{\oa d})$ defines a finite-type affine variety \(Z\) over \(W\coloneqq\mathrm{Spec}\bbC[\bal,\bbe,\bga]\). Statement (i) is equivalent to the statement that the geometric generic fiber
    \[
     Z\times_W \mathrm{Spec}\,\overline{\mathbb{C}(\bal, \bbe, \bga)}
    \]
    is non-empty, which is also equivalent to the statement that the generic fiber
    \[
    Z\times_W \mathrm{Spec}\,{\mathbb{C}(\bal, \bbe, \bga)}
    \]
    is non-empty, as the former is a basechange of the latter. Using textbook facts, see \cite[Lemmas~37.24.1 and ~37.24.2]{Sta}, the last statement is equivalent to the statement that the general fiber of \(Z\to W\) is non-empty. This is equivalent to (ii).
\end{proof}

\subsection{Proof of Lemma~\ref{l:GWV-HNP}} \label{ss:mainPf}
We have shown $\mathcal {S}(\widehat{u},\widehat{v},\widehat{w}w',{\oa d})$ has size $O(n^5(n+\dSize))$. We see $\mathcal{T}(\bfd)$ has $O(n^2)$ equations, with its solutions bounded between $0$ and $D$.
Thus, deciding if $\mathcal {S}(\widehat{u},\widehat{v},\widehat{w}w',{\oa d})$ is satisfiable over $\overline{\mathbb{C}(\bal, \bbe, \bga)}$ for some solution ${\oa d}$ to $\mathcal{T}(\bfd)$ is an instance of $\HNPE$.

We assume~\eqref{eqn:dimcount} holds.
By Proposition~\ref{prop:GWsystem}, $\mathcal {S}(\widehat{u},\widehat{v},\widehat{w}w',{\oa d})$ is satisfiable for a generic choice of evaluations $\oa \al, \oa \be, \oa \ga$
for some solution ${\oa d}$ to $\mathcal{T}(\bfd)$ if and only if $\hgw{[Y_{[u]}],[Y_{[v]}],[Y_{[w]}]}{}{\bfd}\neq 0$.
Combining this with Lemma~\ref{prop:genericRed}, $\mathcal {S}(\widehat{u},\widehat{v},\widehat{w}w',{\oa d})$ is satisfiable over $\overline{\mathbb{C}(\bal, \bbe, \bga)}$ for some for some solution ${\oa d}$ to $\mathcal{T}(\bfd)$ if and only if $\hgw{[Y_{[u]}],[Y_{[v]}],[Y_{[w]}]}{}{\bfd}\neq 0$. Thus
$\neg\GWV$ \. reduces to \. $\HNPE$. \qed

\medskip

\section{Final Remarks}\label{sec:finRem}

\subsection{\(m\)-pointed fixed-domain GW invariants}\label{ss:finrem-m}
Using a similar argument as in~\cite{PR25} for $m$-fold intersections of Schubert varieties, it is straightforward to apply our method to genus zero, \(m\)-pointed fixed-domain GW invariants. These invariants appear as coefficients when one expands a product of \(m-1\) Schubert classes into a linear combination of Schubert classes in the small quantum cohomology ring of \(F(\mathbf{a},n)\), see \cite[$\S$4]{CF99b}.

\subsection{Extending to other Lie types}

One may be able to replace partial flag varieties with generalized (partial) flag varieties of classical Lie types, by considering the corresponding analogs of hyperquot schemes, see \cite{CCH21,CCH22,Sin24}.

\subsection{Grassmannian case}
Specializing to the Grassmannian,  by \cite{AW98} and \cite{Bel01}, the non-vanishing of Gromov--Witten invariants controls the multiplicative Horn problem for ${\sf SU}(n)$, i.e. how the equation $A\cdot B=C$ constrains the singular values of $A,B,C\in {\sf SU}(n)$. In this setting, GW invariants can be computed via classical Schubert calculus on a two-step flag variety \cite{BKT03}. The puzzle rule of \cite{BKPT16} ensures that deciding the non-vanishing of GW invariants is in ${\sf NP}$.
As noted in~\cite[Question~2]{ARY19}, it is an open problem to prove the vanishing is in ${\sf P}$. We note that in this case, GW invariants have the saturation property, as proven by~\cite{Bel08}, but it is unknown if this fact may be leveraged in analogy with~\cite{DeLM06,MNS12} to prove the vanishing problem is in ${\sf P}$.

\subsection{Complexity implications}
It is worth noting that the Main Theorem~\ref{t:main-AM} shows that the
vanishing of GW invariants is rather low in the polynomial hierarchy.
In particular, the result of Boppana, H{\aa}stad and Zachos \cite[Thm~2.3]{BHZ87}
implies that \ts $\GWV$ \ts is not $\NP$-hard, assuming the $\GRH$ and the
polynomial hierarchy
does not collapse to the second level:  $\Sigmap_2 \ne \PH$.  Additionally,
we have that \ts $\GWV \in \coNP$ \ts assuming the $\GRH$ and the
strong derandomization assumption of Impagliazzo and Wigderson \cite{IW97}.

Finally, we note that in \cite{PR25} the first two authors proved
a stronger inclusion for the $\bfd=0$ case, that the vanishing of Schubert structure constants
is in \ts $\AM\cap\coAM$ assuming the $\GRH$.
Unfortunately the tools in that paper do not extend to Gromov--Witten
invariants.

\vskip.4cm
{\small
\subsection*{Acknowledgements}
We thank Anders Buch, Sergei Gukov, Leonardo Mihalcea and Ravi Vakil
for interesting discussions and helpful comments.
This project began when the second and third authors were participants
at the Workshop on Combinatorics of Enumerative Geometry at the
Institute of Advanced Study in Princeton,~NJ.  We are grateful for the hospitality.
The first author was partially supported by the NSF grant CCF-2302173.
The second author was partially supported by the NSF MSPRF grant DMS-2302279.
}

\newpage


\end{document}